\documentclass[10pt]{amsart}
\usepackage{enumerate,amsmath,amssymb,latexsym,
amsfonts, amsthm, amscd, MnSymbol}

\theoremstyle{plain}

\newtheorem{theorem}{Theorem}

\newtheorem*{theorem*}{Theorem}

\numberwithin{equation}{section}

\newcommand{\R}{\mathbb{R}}

\newcommand{\rd}{{\rm d}}

\newcommand{\ii}{{\rm i}}
\newcommand{\dd}{{\rm dn}_2}
\newcommand{\ddd}{{\rm dn}_3}

\newcommand{\dr}{\rm d}

\begin{document}

\title {Modular transformations and \\ the elliptic functions of Shen}

\date{}

\author[P.L. Robinson]{P.L. Robinson}

\address{Department of Mathematics \\ University of Florida \\ Gainesville FL 32611  USA }

\email[]{paulr@ufl.edu}

\subjclass{} \keywords{}

\begin{abstract}
We employ Weierstrassian modular transformations to compute fundamental periods for the elliptic functions $\dd$ and $\ddd$ of Shen. 
\end{abstract}

\maketitle

\section*{An introduction} 

\medbreak 

Ramanujan's theories of elliptic functions to alternative bases were provided with specific elliptic functions by Li-Chien Shen: an elliptic function $\ddd$ [2004] in signature three and an elliptic function $\dd$ [2014] in signature four. The definition of each of these functions involves inverting an incomplete hypergeometric integral on the real line; in each case, the resulting function is seen to satisfy a differential equation whose solutions are known to be elliptic. 

\medbreak 

When an elliptic function arises as a solution to a differential equation, its periods are often expressed as integrals. Archetypically, when a Weierstrass $\wp$ function appears as a solution to 
$$(f')^2 = 4 f^3 - g_2 f - g_3$$
with $g_2$ and $g_3$ real, its real fundamental half-period has the form 
$$\int_{e_1}^{\infty} (4 t^3 - g_2 t - g_3)^{- \tfrac{1}{2}} \; {\rm d} t$$
where $e_1$ is the largest zero of the cubic $4 t^3 - g_2 t - g_3$; its imaginary fundamental half-period has a similar integral expression. 

\medbreak 

Because $\ddd$ and $\dd$ are recognized as solutions to differential equations, their fundamental half-periods may be expressed in the way just described; recasting them hypergeometrically often calls for tricky and seemingly {\it ad hoc} manipulations. Our purpose here is to show that the half-periods of $\ddd$ and $\dd$ may be expressed in explicit hypergeometric terms quite simply and indeed naturally. 

\medbreak 

On the one hand, the hypergeometric origins of $\ddd$ and $\dd$ give immediate explicit form to their real fundamental half-periods, on account of the standard integral identity 
$$\int_0^{\frac{1}{2} \pi} F(a, b; \tfrac{1}{2}; \kappa^2 \sin^2 t) \; {\rm d} t = \tfrac{1}{2} \pi \, F(a, b; 1; \kappa^2).$$
On the other hand, their imaginary fundamental half-periods may also be given explicit hypergeometric form without the need for further integration: we show that they may be derived from the real fundamental half-periods by the use of Weierstrassian modular transformations that are associated to trimidiation and dimidiation. 

\medbreak

\section*{Two modular transformations} 

\medbreak 

We prepare our analysis of the elliptic functions $\dd$ and $\ddd$ by assembling certain facts regarding modular transformations as they pertain to Weierstrass $\wp$ functions. Thus, let $p$ be a Weierstrass $\wp$ function: specifically, let it be the Weierstrass function with invariants $g_2$ and $g_3$; as an alternative description, let it be the Weierstrass function having $(2 \omega, 2 \omega ')$ as a fundamental pair of periods. We may name $p$ by its invariants or by its half-periods, writing 
$$p = \wp (\bullet ; g_2, g_3) = \wp (\bullet ; \omega, \omega ').$$  

\medbreak 

The Weierstrass function 
$$q = \wp (\bullet ; \omega, \tfrac{1}{n} \omega ') = \wp (\bullet ; h_2, h_3)$$ 
obtained from $p$ upon division of a period by the positive integer $n$ is said to arise from $p$ via a modular transformation. We are especially interested in the effect of such a modular transformation on the invariants of a Weierstrass function: that is, we wish to determine the invariants $h_2$ and $h_3$ of $q$ in terms of the invariants $g_2$ and $g_3$ of $p$; in fact, we shall only require this information in the cases $n = 2$ and $n = 3$. In each case we merely state the results, referring to [1973] and [1989] for proofs.  

\medbreak 

The effect of a quadratic transformation is as follows. Here, $q = \wp (\bullet ; \omega, \tfrac{1}{2} \omega ').$ 

\medbreak 

\begin{theorem} \label{n=2}
If $n = 2$ and $b = p(\omega ')$ then 
$$h_2 = 60 b^2 - 4 g_2$$
and 
$$h_3 = 56 b^3 + 8 g_3.$$ 
\end{theorem} 

\begin{proof} 
This proceeds from an inspection of the identity 
$$q(z) = p(z) + p(z + \omega ') - p(\omega ').$$
Details of the derivation may be found in [1973] Section 65 and in [1989] Section 9.8. 
\end{proof} 

\medbreak 

The effect of a cubic transformation is as follows. Here, $q = \wp (\bullet ; \omega, \tfrac{1}{3} \omega ').$

\medbreak 

\begin{theorem} \label{n=3}
If $n = 3$ and $b = p (\tfrac{2}{3} \omega ')$ then 
$$h_2 = 120 b^2 - 9 g_2$$ 
and 
$$h_3 = 280 b^3 - 42 b g_2 - 27 g_3.$$ 
\end{theorem} 

\begin{proof} 
This proceeds from an inspection of the identity 
$$q(z) = p(z) + p(z + \tfrac{2}{3} \omega ') + p(z - \tfrac{2}{3} \omega ') - 2 p(\tfrac{2}{3} \omega ').$$ 
Details of the derivation may be found in Section 68 of [1973]; see also Exercises 8 and 9 of [1989] Chapter 9. 
\end{proof} 

\medbreak 

\section*{Signature three}

\medbreak 

We begin by briefly reviewing the origin of the elliptic function $\ddd$. For further details, we refer the reader to [2004]. 

\medbreak 

Fix $\kappa \in (0, 1)$ as modulus and $\lambda = \sqrt{1 - \kappa^2}$ as complementary modulus. Write $\phi: \R \to \R$ for the inverse to the strictly increasing surjective function 
$$\R \to \R : T \mapsto \int_0^T F(\tfrac{1}{3}, \tfrac{2}{3} ; \tfrac{1}{2} ; \kappa^2 \sin^2 t) \, {\rm d}t$$ 
and write 
$$K = \int_0^{\frac{1}{2} \pi} F(\tfrac{1}{3}, \tfrac{2}{3} ; \tfrac{1}{2} ; \kappa^2 \sin^2 t) \, {\rm d} t = \tfrac{1}{2} \pi  F(\tfrac{1}{3}, \tfrac{2}{3} ; 1 ; \kappa^2).$$
Elementary calculations show that $\phi$ satisfies 
$$\phi(u + 2 K) = \phi (u) + \pi$$ 
whence its derivative $\phi \,': \R \to \R$ has (least positive) period $2 K$. We shall write $\delta = \phi \,'$ for this derivative, writing $\delta_{\kappa}$ when we wish to draw attention to the modulus $\kappa$. 

\medbreak 

The function $\delta$ satisfies the initial condition $\delta(0) = 1$ because the function inverse to $\phi$ plainly has derivative $1$ at the origin; with rather more work, it may be shown that $\delta$ satisfies the differential equation 
$$9 (\delta')^2 = 4 (1 - \delta) (\delta^3 + 3 \delta^2 - 4 \lambda^2).$$ 
As the right-hand side of this complex differential equation is a quartic with simple zeros, its solutions are elliptic functions; the specific solution with $\delta(0) = 1$ is singled out as follows. 

\medbreak 

\begin{theorem} \label{dn3}
The function $\delta_{\kappa} = \phi \,'$ satisfies  
$$(1 - \delta_{\kappa}) (\tfrac{1}{3} + p_{\kappa}) = \tfrac{4}{9} \kappa^2$$ 
where $p_{\kappa} = \wp(\bullet; g_2, g_3)$ is the Weierstrass function with invariants 
$$g_2 = \tfrac{4}{27} (9 - 8 \kappa^2) = \tfrac{4}{27} (8 \lambda^2 + 1)$$ 
and 
$$g_3 = \tfrac{8}{729} (27 - 36 \kappa^2 + 8 \kappa^4) = \tfrac{8}{729} (8 \lambda^4 + 20 \lambda^2 - 1).$$ 
\end{theorem} 

\begin{proof} 
See [2004]; the proof is effected by reference to Section 20.6 of the classic [1927]. 
\end{proof} 

\medbreak 

Thus, $\delta$ is the restriction to $\R$ of an elliptic function; this elliptic extension of $\delta$ is the function $\ddd$ of Shen.

\medbreak 

The elliptic function $\ddd = \delta_{\kappa}$ and the Weierstrass function $p_{\kappa}$ are plainly coperiodic. We shall denote by $(2 \omega_{\kappa}, 2 \omega_{\kappa} ')$ their fundamental pair of periods for which $\omega_{\kappa}$ and $- \ii \, \omega_{\kappa} '$ are strictly positive; we may then also write $p_{\kappa} = \wp(\bullet; \omega_{\kappa}, \omega_{\kappa} ')$. We have already identified the real half-period $\omega_{\kappa}$: it is precisely $K$ as displayed above; that is, 
$$\omega_{\kappa} = \tfrac{1}{2} \pi  F(\tfrac{1}{3}, \tfrac{2}{3} ; 1 ; \kappa^2).$$
We now proceed to evaluate the imaginary half-period $\omega_{\kappa}'$. A customary method for performing such an evaluation involves the calculation of an integral. We propose to depart from this custom: instead, we shall make use of a modular transformation of the sort that is appropriate to Weierstrass functions. 

\medbreak 

Explicitly, alongside the Weierstrass function $p_{\kappa} = \wp(\bullet; \omega_{\kappa}, \omega_{\kappa} ')$ we introduce the Weierstrass function 
$$q_{\kappa} = \wp (\bullet; \omega_{\kappa}, \tfrac{1}{3} \omega_{\kappa} ')$$ 
that results upon division of its imaginary period by three. 

\medbreak 

\begin{theorem} \label{qdn3}
The Weierstrass functions $p$ and $q$ are related by 
$$q_{\kappa} (z) = - 3 \, p_{\lambda} (\sqrt3 \ii \, z).$$ 
\end{theorem} 

\begin{proof} 
Note the passage to the complementary modulus. First, apply Theorem \ref{n=3} and take into account the fact that 
$$b = p_{\kappa}(\tfrac{2}{3} \omega_{\kappa} ') = - \tfrac{1}{3};$$ 
this nontrivial fact is proved in Section 5 of [2004]. From the $\kappa$-dependent formulae for $g_2$ and $g_3$ in Theorem \ref{dn3} it follows by substitution that the invariants $h_2$ and $h_3$ of $q_{\kappa}$ are given by 
$$h_2 = 120 b^2 - 9 g_2 = \tfrac{4}{3} (1 + 8 \kappa^2)$$ 
and 
$$h_3 = 280 b^3 - 42 b g_2 - 27 g_3 = \tfrac{8}{27} (1 - 20 \kappa^2 - 8 \kappa^4).$$ 
It is now convenient to write $g_2(f)$ and $g_3 (f)$ for the quadrinvariant and cubinvariant of any Weierstrass function $f$. With this understanding, we have just established that 
$$g_2 (q_{\kappa}) = 9 \, g_2 (p_{\lambda}) = (\sqrt3 \ii)^4 g_2 (p_{\lambda})$$
and 
$$g_3 (q_{\kappa}) = - 27 \, g_3 (p_{\lambda}) = (\sqrt3 \ii)^6 g_3 (p_{\lambda})$$ 
by reference to Theorem \ref{dn3} for the complementary modulus. The homogeneity relation for Weierstrass functions carries us to the announced conclusion 
$$q_{\kappa} (z) = (\sqrt3 \ii)^2 \, p_{\lambda} (\sqrt3 \ii \, z).$$
\end{proof} 

\medbreak 

This relationship between Weierstrass functions entails a relationship between their half-periods. Explicitly, $q_{\kappa}$ has fundamental half-periods $\omega_{\kappa}$ and $\tfrac{1}{3} \omega_{\kappa} '$ while $p_{\lambda}$ has fundamental half-periods $\omega_{\lambda}$ and $\omega_{\lambda}'$. Accordingly, we deduce the relationship 
$$\omega_{\kappa} ' = \sqrt3 \, \ii \, \omega_{\lambda}.$$ 

\medbreak 

\begin{theorem} \label{per3}
The fundamental half-periods of $\ddd = \delta_{\kappa}$ and $p_{\kappa}$ are given by 
$$\omega_{\kappa} = \tfrac{1}{2} \pi \, F(\tfrac{1}{3}, \tfrac{2}{3} ; 1 ; \kappa^2)$$ 
and 
$$\omega_{\kappa}' = \ii \tfrac{\sqrt3}{2} \pi \, F(\tfrac{1}{3}, \tfrac{2}{3} ; 1 ; 1 - \kappa^2).$$ 
\end{theorem} 

\begin{proof} 
The real half-period has already been identified; the imaginary half-period follows at once from the relationship displayed immediately prior to the present Theorem, on account of the fact that $\lambda^2 = 1 - \kappa^2$. 
\end{proof} 

\medbreak 

Thus the shape of the period lattice is given by the period ratio
$$\frac{\omega_{\kappa}'}{\omega_{\kappa}} = \ii \sqrt3 \, \frac{F(\tfrac{1}{3}, \tfrac{2}{3} ; 1 ; 1 - \kappa^2)}{F(\tfrac{1}{3}, \tfrac{2}{3} ; 1 ; \kappa^2)}\, .$$ 

\medbreak 

\section*{Signature four} 

\medbreak 

We begin by briefly reviewing the origin of the elliptic function $\dd$. For further details, we refer the reader to [2014]. 

\medbreak 

Fix $\kappa \in (0, 1)$ as modulus and $\lambda = \sqrt{1 - \kappa^2}$ as complementary modulus. Write $\phi: \R \to \R$ for the inverse to the strictly increasing surjective function 
$$\R \to \R : T \mapsto \int_0^T F(\tfrac{1}{4}, \tfrac{3}{4} ; \tfrac{1}{2} ; \kappa^2 \sin^2 t) \, {\rm d}t$$ 
and write 
$$K = \int_0^{\frac{1}{2} \pi} F(\tfrac{1}{4}, \tfrac{3}{4} ; \tfrac{1}{2} ; \kappa^2 \sin^2 t) \, {\rm d} t = \tfrac{1}{2} \pi  F(\tfrac{1}{4}, \tfrac{3}{4} ; 1 ; \kappa^2).$$
Elementary calculations show that $\phi$ satisfies 
$$\phi(u + 2 K) = \phi (u) + \pi$$ 
whence if 
$$\psi = \arcsin (\kappa \sin \phi)$$
then the function $\dr = \cos \psi$ has (least positive) period $2 K$. When we wish to place the modulus $\kappa$ in evidence, it shall appear as a subscript. 

\medbreak

The function $\rd$ satisfies the initial condition $\rd (0) = 1$ quite plainly; less plainly, it also satisfies the differential equation 
$$(\rd ')^2 = 2 (1 - \rd) (\rd^2 - \lambda^2).$$ 
The solution to this initial value problem extends to an elliptic function that may be expressed in terms of its coperiodic Weierstrass $\wp$ function, as follows. 

\medbreak 

\begin{theorem} \label{dn2}
The function $\rd_{\kappa} = \cos \psi$ satisfies 
$$(1 - \rd_{\kappa}) (\tfrac{1}{3} + p_{\kappa}) = \tfrac{1}{2} \kappa^2$$ 
where $p_{\kappa} = \wp(\bullet; g_2, g_3)$ is the Weierstrass function with invariants 
$$g_2 = \tfrac{4}{3} - \kappa^2 = \lambda^2 + \tfrac{1}{3}$$ 
and 
$$g_3 = \tfrac{8}{27} - \tfrac{1}{3} \kappa^2 = \tfrac{1}{3} \lambda^2 - \tfrac{1}{27}.$$ 
\end{theorem} 

\begin{proof} 
See [2014]; again, the proof refers to Section 20.6 of [1927]. 
\end{proof} 

\medbreak 

The function $\dd$ of Shen is the elliptic extension of $\rd$ guaranteed by this Theorem. 

\medbreak 

We write $(2 \omega_{\kappa}, 2 \omega_{\kappa} ')$ for the fundamental pair of periods for $\dd = \rd_{\kappa}$ and $p_{\kappa}$ such that $\omega_{\kappa}$ and $- \ii \omega_{\kappa} '$ are strictly positive. The real half-period $\omega_{\kappa}$ has already been identified in hypergeometric terms; the imaginary half-period $\omega_{\kappa} '$ will now be similarly  identified by means of an appropriate modular transformation. 

\medbreak 

Thus, as a companion to $p_{\kappa} = \wp (\bullet ; \omega_{\kappa}, \omega_{\kappa} ')$ we introduce the Weierstrass function 
$$q_{\kappa} = \wp (\bullet; \omega_{\kappa}, \tfrac{1}{2} \omega_{\kappa} ')$$ 
that results upon halving its imaginary period. 

\medbreak 

\begin{theorem} \label{qdn2}
The Weierstrass functions $p$ and $q$ are related by 
$$q_{\kappa} (z) = - 2 \, p_{\lambda} (\sqrt2 \ii \, z).$$ 
\end{theorem} 

\begin{proof} 
Again, note the involvement of the complementary modulus. The proof follows the line of argument for Theorem \ref{qdn3}. It is shown in Section 4 of [2014] that 
$$p_{\kappa} (\omega_{\kappa} ') = - \tfrac{1}{3}.$$ 
Accordingly, by application of Theorem \ref{n=2} along with reference to the $\kappa$-dependent formulae for $g_2$ and $g_3$ in Theorem \ref{dn2}, the invariants $h_2$ and $h_3$ of $q_{\kappa}$ are found to be  
$$h_2 = 4(\tfrac{1}{3} + \kappa^2)$$
(which is $4 = (\sqrt2 \ii)^4$ times the quadrinvariant of $p_{\lambda}$) and 
$$h_3 = 8(\tfrac{1}{27} - \tfrac{1}{3} \kappa^2)$$
(which is $- 8 = (\sqrt2 \ii)^6$ times the cubinvariant of $p_{\lambda}$). Finally, the Weierstrassian homogeneity relation serves to conclude the proof. 
\end{proof} 

\medbreak 

As was the case in signature three, a relation between real and imaginary half-periods for complementary moduli follows here: thus 
$$\omega_{\kappa} ' = \sqrt2 \ii \, \omega_{\lambda}.$$ 

\medbreak 

\begin{theorem} \label{per2}
The fundamental half-periods of $\dd = \rd_{\kappa}$ and $p_{\kappa}$ are given by 
$$\omega_{\kappa} = \tfrac{1}{2} \pi \, F(\tfrac{1}{4}, \tfrac{3}{4} ; 1 ; \kappa^2)$$ 
and 
$$\omega_{\kappa}' = \ii \tfrac{\sqrt2}{2} \pi \, F(\tfrac{1}{4}, \tfrac{3}{4} ; 1 ; 1 - \kappa^2).$$ 
\end{theorem} 

\begin{proof} 
As for signature three, the relationship between $\omega_{\kappa}'$ and $\omega_{\lambda}$ allows us to derive the imaginary half-period at once from the previously identified real half-period. 
\end{proof} 

\medbreak 

The shape of the period lattice is thus given by the period ratio
$$\frac{\omega_{\kappa}'}{\omega_{\kappa}} = \ii \sqrt2 \, \frac{F(\tfrac{1}{4}, \tfrac{3}{4} ; 1 ; 1 - \kappa^2)}{F(\tfrac{1}{4}, \tfrac{3}{4} ; 1 ; \kappa^2)}\, .$$

\bigbreak

\begin{center} 
{\small R}{\footnotesize EFERENCES}
\end{center} 
\medbreak

[1927] E.T. Whittaker and G.N. Watson, {\it A Course of Modern Analysis}, Fourth Edition, Cambridge University Press. 

\medbreak 

[1973] P. Du Val, {\it Elliptic Functions and Elliptic Curves}, L.M.S. Lecture Note Series {\bf 9}, Cambridge University Press. 

\medbreak 

[1989] D.F. Lawden, {\it Elliptic Functions and Applications}, Applied Mathematical Sciences {\bf 80}, Springer-Verlag. 

\medbreak 

[2004] Li-Chien Shen, {\it On the theory of elliptic functions based on $_2F_1(\frac{1}{3}, \frac{2}{3} ; \frac{1}{2} ; z)$}, Transactions of the American Mathematical Society {\bf 357} 2043-2058. 

\medbreak 

[2014] Li-Chien Shen, {\it On a theory of elliptic functions based on the incomplete integral of the hypergeometric function $_2 F_1 (\frac{1}{4}, \frac{3}{4} ; \frac{1}{2} ; z)$}, Ramanujan Journal {\bf 34} 209-225. 

\medbreak

\end{document}